\documentclass[a4paper, draft, 12pt]{amsproc}
\usepackage{amssymb}
\usepackage{amsmath,amssymb,ulem,color}

\usepackage[hyphens]{url} 
\usepackage[dvips]{graphicx} 
\usepackage{cmdtrack} 
\usepackage{mathrsfs}                       

\theoremstyle{plain}
 \newtheorem{theorem}{Theorem}[section]
 
 \newtheorem{lem}{Lemma}[section]
 
\theoremstyle{Definition}
 \newtheorem{exm}{Example}[section]
 
 \newtheorem{dfn}{Definition}[section]
\theoremstyle{remark}
 \newtheorem{rem}{Remark}[section]
 
 \numberwithin{equation}{section}

\renewcommand{\geq}{\geqslant}

\title[A Study of Fatou Set, Julia set and Escaping Set in  Nearly Abelian...]{A Study of Fatou Set, Julia set and Escaping Set in Nearly Abelian Transcendental Semigroup}



\author[B. H. Subedi]{\bfseries  Bishnu Hari Subedi}

\address{ 
Central Department of Mathematics \\ 
Institute of Science and Technology   \\ 
Tribhuvan University   \\ 
Kirtipur, Kathmandu\\
Nepal}
\email{subedi.abs@gmail.com / subedi\_bh@cdmathtu.edu.np }

\author[A. Singh]{Ajaya Singh}
\address{Central Department of Mathematics, Institute of Science and Technology, Tribhuvan University, Kirtipur, Kathmandu, Nepal }
\email{singh.ajaya1@gmail.com / singh\_a@cdmathtu.edu.np} 
\thanks{Supported by ... } 
\thanks{This research work of first author is supported from PhD faculty fellowship of University Grants Commission, Nepal. } 

\begin{document}

{\begin{flushleft}\baselineskip9pt\scriptsize
MANUSCRIPT
\end{flushleft}}
\vspace{18mm} \setcounter{page}{1} \thispagestyle{empty}

\begin{abstract}
We mainly generalize the notion of abelian transcendental semigroup to nearly abelian transcendental semigroup. We prove that Fatou set, Julia set and escaping set of nearly abelian transcendental semigroup are completely invariant. We investigate no wandering domain theorem in such a transcendental semigroup. We also obtain results on a complete generalization of the classification of  periodic Fatou components.
\end{abstract}

\maketitle

\section{Introduction}
\subsection{A short review of classical transcendental dynamics}
Throughout this paper, we denote the \textit{complex plane} by $\mathbb{C}$ and set of integers greater than zero by $\mathbb{N}$. 
We assume the function $f:\mathbb{C}\rightarrow\mathbb{C}$ is \textit{transcendental entire function} unless otherwise stated. 
For any $n\in\mathbb{N}, \;\; f^{n}$ always denotes the nth \textit{iterates} of $f$. Let $ f $ be a transcendental entire function. The set of the form
$$
I(f) = \{z\in \mathbb{C}:f^n(z)\rightarrow \infty \textrm{ as } n\rightarrow \infty \}
$$
is called an \textit{escaping set} and any point $ z \in I(S) $ is called \textit{escaping point}. For transcendental entire function $f$, the escaping set $I(f)$ was first studied by A. Eremenko \cite{ere}. He showed that 
 $I(f)\not= \emptyset$; the boundary of this set is a Julia set $ J(f) $ (that is, $ J(f) =\partial I(f) $);
 $I(f)\cap J(f)\not = \emptyset$; and 
 $\overline{I(f)}$ has no bounded component. By motivating from this last statement, he posed a question: \textit{Is every component of $ I(f) $ unbounded?}. This question is considered  as an important open problem of transcendental dynamics and nowadays is famous as \textit{Eremenko's conjecture}. Note that the complement of Julia set $ J(f) $ in complex plane $ \mathbb{C} $ is a \textit{Fatou set} $F(f)$.
 
 Recall that the set $ C(f) = \{z\in \mathbb{C} : f^{\prime}(z) = 0 \}$ is the set of \textit{critical points} of the transcendental entire function $ f $ and  the set $CV(f) = \{w\in \mathbb{C}: w = f(z)\;\ \text{such that}\;\ f^{\prime}(z) = 0\} $ of all images of all critical points is called the set of \textit{critical values}.  The set 
$AV(f)$ consisting of all  $w\in \mathbb{C}$ such that there exists a curve (asymptotic path) $\Gamma:[0, \infty) \to \mathbb{C}$ so that $\Gamma(t)\to\infty$ and $f(\Gamma(t))\to w$ as $t\to\infty$ is called the set of \textit{asymptotic values} of $ f $ and the set
$SV(f) =  \overline{(CV(f)\cup AV(f))}$
is called the \textit{singular values} of $ f $.  
If $SV(f)$ has only finitely many elements, then $f$ is said to be of \textit{finite type}. If $SV(f)$ is a bounded set, then $f$ is said to be of \textit{bounded type}.                                                                                                                             
The sets
$$\mathscr{S} = \{f:  f\;\  \textrm{is of finite type}\} 
\;\;  \text{and}\; \;                                                                                                                                                                                                                                                                                                                                                                                                                                                                                                                                                                                                                                                                                                                                                                                                                                                                                                                                                                                                                                                                                                                                                                                                                                                                                                                                                                                                                                                                                                                                                                                                                                                                                                                                                                                                                                                                                                                                                                                                                                                                                                                                                                                                                                                                                                                                                                                                                                                                                                                                                                                                                                                                                                                                                                                                                                                                                                                                                                                                                                                                                                                     
\mathscr{B} = \{f: f\;\  \textrm{is of bounded type}\}
$$
are respectively called \textit{Speiser class} and \textit{Eremenko-Lyubich class}. 

\subsection{Brief review of transcendental semigroup dynamics}

We confine our study on Fatou set, Julia set and escaping set of transcendental semigroup. It is very obvious fact that a set of transcendental entire maps on $ \mathbb{C} $  naturally forms a semigroup. Here, we take a set $ A $ of transcendental entire maps and construct a semigroup $ S $ consists of all elements that can be expressed as a finite composition of elements in $ A $. We call such a semigroup $ S $ by \textit{transcendental semigroup} generated by set $ A $. A non-empty subset $ T $ of holomorphic semigroup $ S $ is a \textit{subsemigroup} of $ S $ if $ f \circ g \in T $ for all $ f, \; g \in T $.  
Our particular interest is to study of the dynamics of the families of transcendental entire maps.  For a collection $\mathscr{F} = \{f_{\alpha}\}_{\alpha \in \Delta} $ of such maps, let 
$$
S =\langle f_{\alpha} \rangle
$$ 
be a \textit{transcendental semigroup} generated by them. The index set $ \Delta $ to which $ \alpha $  belongs is allowed to be infinite in general unless otherwise stated. 
Here, each $f \in S$ is a transcendental entire function and $S$ is closed under functional composition. Thus, $f \in S$ is constructed through the composition of finite number of functions $f_{\alpha_k},\;  (k=1, 2, 3,\ldots, m) $. That is, $f =f_{\alpha_1}\circ f_{\alpha_2}\circ f_{\alpha_3}\circ \cdots\circ f_{\alpha_m}$. 

A semigroup generated by finitely many transcendental functions $f_{i}, (i = 1, 2, \\  \ldots, n) $  is called \textit{finitely generated transcendental semigroup}. We write $S= \langle f_{1}, f_{2},\\  \ldots,f_{n} \rangle$.
 If $S$ is generated by only one transcendental entire function $f$, then $S$ is \textit{cyclic transcendental semigroup}. We write $S = \langle f\rangle$. In this case, each $g \in S$ can be written as $g = f^n$, where $f^n$ is the nth iterates of $f$ with itself. Note that in our study of  semigroup dynamics, we say $S = \langle f\rangle$  a \textit{trivial transcendental semigroup}. The transcendental semigroup $S$ is \textit{abelian} if  $f_i\circ f_j =f_j\circ f_i$  for all generators $f_{i}$ and $f_{j}$  of $ S $. The transcendental semigroup $ S $ is \textit{bounded type (or finite type)} if each of its generators $ f_{i} $ is bounded type (or finite type).
 
The family $\mathscr{F}$  of complex analytic maps forms a \textit{normal family} in a domain $ D $ if given any composition sequence $ (f_{\alpha}) $ generated by the member of  $ \mathscr{F} $,  there is a subsequence $( f_{\alpha_{k}}) $ which is uniformly convergent or divergent on all compact subsets of $D$. If there is a neighborhood $ U $ of the point $ z\in\mathbb{C} $ such that $\mathscr{F} $ is normal family in $U$, then we say $ \mathscr{F} $ is normal at $ z $. If  $\mathscr{F}$ is a family of members from the transcendental semigroup $ S $, then we simply say that $ S $ is normal in the neighborhood of $ z $ or $ S $ is normal at $ z $. 

Let  $ f $ be a transcendental entire map. We say that  $ f $ \textit{iteratively divergent} at $ z \in \mathbb{C} $ if $  f^n(z)\rightarrow \infty\; \textrm{as} \; n \rightarrow \infty$. A sequence $ (f_{k})_{k \in \mathbb{N}} $ of transcendental entire maps is said to be \textit{iteratively divergent} at $ z $ if $ f_{k}^{n}(z) \to\infty \;\ \text{as}\;\ n\to \infty$ for all $ k \in \mathbb{N} $.  Semigroup $ S $ is \textit{iteratively divergent} at $ z $ if $f^n(z)\rightarrow \infty \; \textrm{as} \; n \rightarrow \infty$. Otherwise, a function $ f  $, sequence $ (f_{k})_{k \in \mathbb{N}} $ and semigroup $ S $  are said to be \textit{iteratively bounded} at $ z $. 

Based  on the Fatou-Julia-Eremenko theory of a complex analytic function, the Fatou set, Julia set and escaping set in the settings of transcendental  semigroup are defined as follows.
\begin{dfn}[\textbf{Fatou set, Julia set and escaping set}]\label{2ab} 
\textit{Fatou set} of the transcendental semigroup $S$ is defined by
  $$
  F (S) = \{z \in \mathbb{C}: S\;\ \textrm{is normal in a neighborhood of}\;\ z\}
  $$
and the \textit{Julia set} $J(S) $ of $S$ is the compliment of $ F(S) $ where as the escaping set of $S$ is defined by 
$$
I(S)  = \{z \in \mathbb{C}: S \;  \text{is iteratively divergent at} \;z \}.
$$
We call each point of the set $  I(S) $ by \textit{escaping point}.        
\end{dfn} 
It is obvious that $F(S)$ is the largest open subset  $\mathbb{C}$ on which the family $\mathscr{F} $ in $S$ (or semigroup $ S $ itself) is normal. Hence its compliment $J(S)$ is a smallest closed set for any  semigroup $S$. Whereas the escaping set $ I(S) $ is neither an open nor a closed set (if it is non-empty) for any transcendental semigroup $S$. Any maximally connected subset $ U $ of the Fatou set $ F(S) $ is called a \textit{Fatou component}.  
        
If $S = \langle f\rangle$, then $F(S), J(S)$ and $I(S)$ are respectively the Fatou set, Julia set and escaping set in classical transcendental dynamics. In this situation we simply write: $F(f), J(f)$ and $I(f)$. 
 
In {\cite[Theorem 4.1]{kum2}} and {\cite[Theorem 2.3]{sub2}}, it was shown that escaping set of transcendental semigroup is in general forward invariant. However in {\cite[Theorem 2.1]{kum1}} and  {\cite[Theorem 2.6]{sub2}}, it was shown under certain condition  that escaping set of transcendental semigroup is backward invariant. We proved under certain condition that Fatou set, Julia set and escaping set of transcendental semigroup respectively equal to the Fatou set, Julia set and escaping set of its subsemigroup {\cite[Theorems 1.1 and 3.1]{sub4}}).  In {\cite[Theorems 3.2 and 3.3]{sub2}}  we proved that  Fatou set, Julia set and escaping set of transcendental semigroup respectively equal to the Fatou set, Julia set and escaping set of each of its function if semigroup $ S $ is abelian and in such case these sets are completely invariant. There is a slightly larger family of transcendental  semigroups that can  fulfill this criteria. We call these semigroups 
nearly abelian and it is considered the more general form than that of abelian semigroups.
\begin{dfn}[\textbf{Nearly abelian semigroup}] \label{1p}
We say that a transcendental semigroup $ S $ is \textit{nearly abelian} if there is a  family  $ \Phi = \{\phi_{i} \} $ of  conformal maps of the form $ az + b $ for some non-zero $ a $ such that
\begin{enumerate}
\item $ \phi_{i}(F(S)) = F(S) $ for all $ \phi_{i}\in \Phi $ and
\item for all $ f, g \in S $, there is a $ \phi \in \Phi $ such that $ f \circ g = \phi \circ g\circ f  $.
\end{enumerate}
\end{dfn}
Note that particular example of nearly abelian semigroup is a abelian semigroups. Abelian semigroup follows trivially from nearly abelian semigroup if we choose $ \phi $ an identity function. The nearly abelian semigroups are the simple examples of semigroups which behave likely the same way as the classical trivial semigroups. In this regards, the chief aim of this paper is to prove the following result which we have considered a strongest result of transcendental semigroup dynamics. 
\begin{theorem}\label{1q}
Let $ S $ be a nearly abelian transcendental semigroup. Then for each $ g\in S $, we have  $ I(S) = I(g) $, $ J(S) = J(g)$ and  $F(S) = F(g) $.
\end{theorem}

The chief consequence of nearly abelian transcendental semigroup is attached with wandering domains and the concept of nearly abelian semigroups is quite useful for the classification of periodic component (stable basin) of the Fatou set $ F(S) $.  
\begin{dfn}[\textbf{Stablizer,  wandering component and  stable domains}]\label{1g}
For a holomorphic  semigroup $ S $, let $ U $ be a component of the  Fatou set $ F(S) $ and $ U_{f} $ be a component of Fatou set containing $ f(U) $ for some $ f\in S $.  The set of the form 
$$S_{U} = \{f\in S : U_{f} = U\}  $$
is called \textit{stabilizer} of $ U $ on $ S $. If $ S_{U} $ is non-empty,  we say that a component $ U $ satisfying  $U_{f} = U  $ is called \textit{stable basin} for  $ S $. The component $ U $ of $ F(S) $ is called wandering if the set $ \{U_{f}: f \in S \} $ contains infinitely many elements. That is, $ U $ is a wandering domain if there is sequence of elements $ \{f_{i}\} $ of $ S $ such that $ U_{f_{i}}  \neq U_{f_{j}}$ for $ i \neq j $. Furthermore, the component $ U $ of $ F(S) $ is called strictly wandering if $U_{f} = U_{g} $ implies $ f =g $. A stable basin $ U $ of a holomorphic semigroup $ S $ is
\begin{enumerate}
\item  \textit{attracting} if it is a subdomain of attracting basin of each $ f\in S_{U} $
\item  \textit{supper attracting} if it is a subdomain of supper attracting basin of each $ f\in S_{U} $
\item  \textit{parabolic} if it is a subdomain of parabolic basin of each $ f\in S_{U} $
\item  \textit{Siegel} if it is a subdomain of Siegel disk of each $ f\in S_{U} $
\item  \textit{Baker} if it is a subdomain of Baker domain  of each $ f\in S_{U} $
\item  \textit{Hermann} if it is a subdomain of Hermann  ring of each $ f\in S_{U} $
\end{enumerate}
\end{dfn}
In classical holomorphic iteration theory, the stable basin is one of the above types but in transcendental iteration theory, the stable basin is not a Hermann because a transcendental entire function does not have Hermann ring {\cite[Proposition 4.2] {hou}}.

Note that for any rational function  $ f $, we always have $ U_{f} = U $. So $ U_{S} $ is non-empty for a rational semigroup $ S $.  However,  if $ f $ is transcendental, it is possible that $ U_{f} \neq U $. So, $ U_{S} $ may be empty for transcendental semigroup $ S $.  Bergweiler and Rohde \cite{ber} proved that $ U_{f} - U $ contains at most one point which is an asymptotic value of $ f $ if  $ f $  is an entire function. 

We prove the following no wandering domain result which is analogous to the rational semigroups {\cite[Theorem 5.1]{hin}}.
\begin{theorem}\label{1r}
Let $ S $ be a nearly abelian semigroup generated by transcendental entire functions of finite or bounded type. Then $ F(S) $ has no wandering domain.
\end{theorem}

We prove the following result regarding the classification  of periodic components of Fatou set $ F(S) $. 
\begin{theorem} \label{ct}
Let $ U $ be a  component of the Fatou set $ F(S) $ of the nearly abelian  semigroup $ S $ generated by transcendental entire functions of finite or bounded type and $ V $ be a subset $ F(S) $ containing in the forward orbit of $ U $. Then $ V $ is attracting, super attracting, parabolic, Siegel  or Baker.                                                                                                                      
\end{theorem}

The organization of this paper is as follows: In section 2, we briefly review notion of nearly abelian transcendental semigroup with suitable examples. We prove existence theorem for nearly abelian transcendental semigroup. We also investigate necessary condition of any transcendental semigroup to be nearly abelian.  In section 3, we mainly prove theorem 1.1. and theorem 1.2.  In section 4, we prove classification theorem (Theorem\ref{ct}) of periodic Fatou component of transcendental semigroup.

\section{The Notion of Nearly Abelian Transcendental Semigroup}

In this section, we extend the results of abelian transcendental  semigroups to more general settings of nearly abelian transcendental semigroups.   The principal feature of nearly abelian rational semigroup  was investigated by Hinkannen and Martin {\cite[Theorem 4.1]{hin}}. In such a case,  they found that  the Julia set $J(S)$ of rational semigroup $ S $ is same as Julia set $ J(f) $ of each $ f\in S $. Indeed, this is generalization  of the result of Fatou \cite{fat} and Julia \cite{jul2} (if rational maps $ f $ and $ g $ are permutable, then they have the same Julia sets) in semigroup settings. However, the corresponding result may not hold in the case of permutable transcendental entire functions but in nearly abelian settings of transcendental semigroup, we found that the corresponding result of Hinkannen and Martin {\cite[Theorem 4.1]{hin}  holds.  

The definition \ref{1p} of nearly abelian transcendental  semigroup looks more restrictive on the affine map of the form $\phi(z) = az + b,\; a\neq 0 $ and this type of function can play the role of semiconjugacy  to certain class of transcendental entire functions. Recall that function $ f $ is (semi) conjugate to the function $ g $ if there is a continuous function $ \phi $ such that $ \phi \circ f = g \circ \phi $. For example, transcendental entire function $ f_{1}(z) = \lambda \cos z $ is semi-conjugate to another transcendental entire function $ f_{2}(z) = -\lambda \cos z $ because there is a function   $ \phi (z) = -z $  such that $ \phi \circ f_{1} = f_{2} \circ \phi $. If there is a transcendental semigroup generated by such type of semi-conjugate functions, then semigroup will more likely to be nearly abelian.  

\begin{theorem}\label{1s1}
Let $ S = \langle f_{1}, f_{2}, \ldots f_{n}, \ldots \rangle $ be a transcendental semigroup and  let $ \phi $ be an entire function of the form $ z \rightarrow a z + b $ for some non zero $ a $ with $ a,\; b \in \mathbb{C} $ such that  $ \phi \circ f_{i} =f_{j} \circ \phi $ for all $f _{i}  $ and $ f_{j}$ with $ i \neq j $. If $ \phi \circ f= g$ for all $f \in S $. Then the transcendental semigroup $ S $ is nearly abelian.
\end{theorem}
To prove this theorem \ref{1s1}, we need the following  lemmas.
\begin{lem}\label{1s2}
Let $ S = \langle f_{1}, f_{2}, \ldots f_{n}, \ldots \rangle $ be a transcendental semigroup and let $ \phi $ be an entire function of the form $ z \rightarrow a z + b $ for some non zero $ a $ with $ a,\; b \in \mathbb{C} $. If $ \phi \circ f_{i} =f_{j} \circ \phi $ for all $f _{i}  $ and $ f_{j}$ with $ i \neq j $,  then $ \phi(F(S)) = F(S)$ and $ \phi(J(S)) = J(S)$. 
\end{lem}
\begin{proof}
First of all, we prove that if $ \phi \circ f_{i} = f_{j} \circ \phi $ for all $ i $ and $ j $ with $ i \neq j$, then 
$ \phi \circ f  = g\circ \phi $ for all $ f,\;  g\in S $.

 Since any $ f,\;  g\in S $ can be written as $ f = f_{i_{1}}\circ f_{i_{2}}\circ \ldots \circ f_{i_{n}} $  and $ g = f_{j_{1}}\circ f_{j_{2}}\circ \ldots \circ f_{j_{n}} $. Now $ \phi \circ f = \phi \circ f_{i_{1}}\circ f_{i_{2}}\circ \ldots \circ f_{i_{n}} = f_{j_{1}}\circ \phi \circ \circ f_{i_{2}}\circ \ldots \circ f_{i_{n}} = \ldots = f_{j_{1}}\circ f_{j_{2}}\circ \ldots \circ f_{j_{n}} \circ \phi = g \circ \phi $. This proves our claim.

Let $ w \in \phi (F(S)) $. Then there is $ z_{0}\in F(S) $ such that $ w = \phi(z_{0}) $.  Let $ U\subset F(S) $ is a neighborhood of $ z_{0} $ such that $ |f(z) - f(z_{0})| < \epsilon/2 $ for all $ z \in U $ and $ f \in S $. This shows that $ f(U) $ has diameter less than $ \epsilon $ for all $ f \in S $. Since function $ \phi $ has bounded first derivative $ a \neq 0 $, so it is a Lipschitz with Lipschitz constant $ k = \sup|\phi ^{'}(z)| =a $. Now for any $ g \in S $, the diameter of $ g(\phi(U)) = \phi(f(U)) $ is less than $ k\epsilon $. Hence $ w = \phi(z_{0}) \in F(S) $. This shows that $ \phi(F(S)) \subset F(S) $.

Next, let $ w \in \phi(J(S)) $. Then $ w = \phi (z_{0} )$ for some $ z_{0} \in J(S) $. Let $ z_{0} $ be a repelling fixed point for some $ f \in S $ but which is not a critical point of $ \phi $, then $ \phi \circ f =g \circ \phi $ gives $ g $ has a fixed point at $ \phi (z_{0}) $ with same multiplier as that of $ f $ at $z_{0}$. Thus $ \phi $ maps repelling fixed points of any $ f \in S $ to repelling fixed points of another $ g \in S $. Since from {\cite[Theorem 4.1 and 4.2]{poo}}, Julia set of transcendental semigroup is perfect and $ J(S) = \overline{\cup_{f \in S} J(f))}$, where repelling periodic points are dense in $ J(f) $ for each $ f \in S $. So by above discussion, it then follows that $ \phi(J(S)) \subset J(S) $. 

Finally, since $ \phi (\mathbb{C}) = \mathbb{C} $. Using this fact in  $ F(S) = \mathbb{C}- J(S) $ and $J(S) = \mathbb{C}- F(S) $, we get $ \phi(F(S)) = \mathbb{C}- \phi(J(S)) $ and $\phi(J(S)) = \mathbb{C}-\phi( F(S)) $. Again using facts $ \phi(J(S)) \subset J(S) $ and $ \phi(F(S)) \subset F(S) $ in $ \phi(F(S)) = \mathbb{C}- \phi(J(S)) $ and $\phi(J(S)) = \mathbb{C}-\phi( F(S)) $ respectively, we will get required opposite inclusions $ F(S) \subset \phi(F(S)) $ and  $ J(S) \subset \phi(J(S)) $. 
\end{proof}
Note that this lemma \ref{1s2} tells us that the first condition $ \phi_{i}(F(S)) = F(S) $ of nearly abelian semigroup can be replaced by (semi) conjugacy relation $ \phi \circ f_{i} =f_{j} \circ \phi $ for all $f _{i}  $ and $ f_{j}$ with $ i \neq j $. This is a way that one can replace  the first condition of the definition. 

\begin{proof}[Proof of the Theorem \ref{1s1}]
The first part for  nearly abelian semigroup follows from the lemma \ref{1s2}.
 
The second part follows from the following simple calculations.
The hypothesis $ \phi \circ f_{i} =f_{j} \circ \phi $ for all $f _{i}  $ and $ f_{j}$ with $ i \neq j $ gives $ f \circ \phi = \phi \circ g  $ for all $ f,\; g \in S $ and from the hypothesis $ \phi \circ f = g $ for all $ f \in S $, we can write  $\phi \circ g \circ f  = f \circ \phi \circ f = f \circ g$ for all $ f,\; g \in S $.

\end{proof}
There are general and particular examples of transcendental entire functions that fulfills the essence of above theorem \ref{1s1} and the semigroup generated by these functions is nearly abelian.
\begin{exm}\label{1s4} 
Let $ \phi $ be an entire function of the form $ z \rightarrow -z + c $ for some $ c \in \mathbb{C} $. Let  $ f $ be a transcendental entire function with $ f\circ \phi =f $ and function $ g $ is defined by $ g = \phi \circ f $. Then functions $ f $ and $ g $ are conjugates  and the semigroup  $ S = \langle f, g \rangle $  generated by these two  functions $ f $ and $ g $ is nearly abelian.
\end{exm}
\begin{proof}[Solution]
Let $ f, \; g $ and $ \phi $ be as in the statement of the question. It is clear that $ \phi^{2} $ = Identity.
Then $g \circ \phi = \phi \circ f  \circ \phi = \phi \circ f$. This proves that functions $ f $ and $ g $ are conjugates. The condition  $ \phi(F(S)) = F(S) $ for all $ \phi\in \Phi $ of the definition of nearly abelian semigroup follows from above lemma \ref{1s2}. The second condition follows from the theorem \ref{1s1}.  More explicitly it follows from the following calculation.
$$
f \circ g = f \circ \phi \circ f = f \circ f  = f^{2} = \phi^{2} \circ f^{2} = \phi \circ \phi \circ f \circ f = \phi \circ g \circ f.
$$
Therefore, the semigroup  $ S = \langle f, g \rangle $  generated by these two  functions $ f $ and $ g $ is nearly abelian. From the fact $ g = \phi \circ f $, we can say  that $\phi$ is not an identity. 

\end{proof}
\begin{exm}
Let $ f(z) = e^{z^{2}} + \lambda $, and $ g = \phi \circ f $ where $ \phi(z) = -z $. Then the semi group $ S = \langle f, g \rangle $ is nearly abelian. Like wise, functions $ f(z) = \lambda \cos z $ and $ g = \phi \circ f $ where $ \phi(z) = -z $ generate the nearly abelian semigroup. 
\end{exm}
\begin{proof}[Solution]
The given functions in the question fulfills all conditions such as $ f\circ \phi =f $, $ \phi^{2} $ = identity as well as  $ \phi  \circ f= g \circ \phi $ of above theorem \ref{1s1} as well as example \ref{1s4}. Therefore, the semigroup $ S =\langle f, g \rangle $ is indeed nearly abelian.
Note that $ \phi \circ f = -f \neq f $, so $\phi$ is not an identity.
\end{proof}
Note that the above example \ref{1s4} is just for a nice  general example of above theorem \ref{1s1} that says there is an nearly abelian transcendental semigroup. Unfortunately, this example is not generating many more examples of transcendental entire functions that can generate transcendental semigroup. Basically, it generates even functions or translation of even functions. For example: If we set $ h(z) = f(z + \frac{c}{2}) $, then $ h(z) = f(z + \frac{c}{2}) = (f \circ \phi) (z) = f(c - z -\frac{c}{2}) =f( \frac{c}{2} -z) = h(-z)  $. That is, $ h $ is an even function. 

Above theorem \ref{1s1} provided a criterion to be a nearly abelian transcendental semigroup. Are there other criteria that help us to make nearly abelian transcendental semigroups and somehow connected to above criterion? 
In the case of rational semigroup $ S $, there is very strong criteria for nearly abelian semigroup due to Hinkkanen and Martin {\cite[Corollary 4.1]{hin}}. This criteria is possible because of Beardon's result {\cite[Theorem 1]{bea1}}. It states the following. \textit{For any  two polynomials  $ f $ and $ g $  and degree of $ f $ is at least two, then $ J(f) = J(g)  $ if and only if there is a map $ \phi(z) =az + b $ with $ |a| =1 $ such that $ f \circ g = \phi \circ g \circ f $}. Our aim in this context to see a transcendental semigroup $ S $ that fulfills this results. 
Similar to Hinkkanen and Martin  {\cite[Corollary 4.1]{hin}}, we have formulated the following theorem.
\begin{theorem}\label{nb1}
Let $ S $ be a transcendental semigroup and suppose  that $ I(f) = I(g) $ for all $ f,\; g \in S $. Then $S $ is nearly abelian semigroup.
\end{theorem}
According to this theorem, the condition  $I(f) = I(g) $ for all $ f,\; g \in S $ is  very strong one that replace both of conditions of the definition \ref{1p} of nearly abelian semigroup.  Indeed, the condition $ I(f) = I(g) $ stated in this theorem \ref{nb1} can be used to obtain  the converse of Baker's question \cite{bak}, namely, \textit{for two distinct permutable transcendental entire functions $ f $ and $ g $, does it follow $ J(f) = J(g) $}? This question is itself a difficult one of classical transcendental dynamics to answer.  The converse of this Baker's question is again difficult to settle down. We expect that this will settle down in nearly abelian transcendental semigroups. We completely characterize first all permutable transcendental entire functions. For the given transcendental entire function $ f $ let us define the following three classes of algebraic structures:
$$
\mathcal{S}(f) =\{g : J(g) = J(f) \},
$$
$$
 \mathcal{C}(f) =\{g : f \circ g = g \circ f \},
 $$
 $$
 \Sigma(f) =\{\phi\in \Phi : \phi(J(f)) = J(f) \},
$$
where $ \Phi $ is a group of conformal isometrics $ \phi(z) = az + b $ with $ |a| = 1$. As Julia set $ J(f) $ of any transcendental function $ f $ is unbounded, each element of $ \Phi $ is translation or rotation or both. Note that if both $ f $ and $ g $ were polynomials, Beardon {\cite[Theorem 1]{bea1}} proved that $\mathcal{S}(f) =\{g : f\circ g = \phi \circ g \circ f \; \text{for some}\; \phi \in \Sigma (f) \}$.  In terms of the result of Fatou \cite{fat}, Julia \cite{jul2}, Baker and Eremenko {\cite[Theorem 1]{bak1}}, Beardon  {\cite [Theorem 2]{bea1}} also proved that each of sets $ \mathcal{S}(f) $ and $\mathcal{C}(f) $ is a semigroup,  $ \mathcal{C}(f)\subset \mathcal{S}(f) $ and $ \mathcal{C}(f) = \mathcal{S}(f) $ when $ \Sigma(f) $ is trivial.  Unfortunately, analogous result may not hold if $ f $ and $ g $ are transcendental entire functions. So we need further complete classification of all pair of permutable transcendental entire functions that have same Julia sets. 
Only known characterization of transcendental entire functions $ f $ and $ g $ that can have same Julia set are as follows: 
\begin{enumerate}
\item  if function $ g = f^{n} $ for some $ n (\geq 2) \in \mathbb{N} $; 
\item if  $ f $ and $ g $ are permutable functions such that $ g(z) = af(z) + b $, where $ a (\neq 0, |a| =1)$ and $ b $ are complex constants; 
\item if  $ f $ and $ g $ are permutable functions and $ p(z) $ be a non-constant polynomial such that $ p(g(z)) = ap(f(z)) + b $, where $ a (\neq 0)$ and $ b $ are complex constants; 
\item if  $ f $ and $ g $ are permutable functions of bounded type;
\item if $ f $ and $ g $ are permutable functions without wandering domains:
\item if $ g(z) = af^{n}(z) + b $, where $ |a| =1 $ and $ b \in \mathbb{C} $;
\item  if $ f^{m}(z) = g^{n}(z) $ for some $m,  n \in \mathbb{N} $.
\end{enumerate}
The functions $ f $ and $ g $ stated above belong to  class $\mathcal{C}(f) $ with same Julia sets and in such functions, we can write $ \mathcal{C}(f) \subset \mathcal{S}(f) $. However,  in general $ \mathcal{C}(f)\cap \mathcal{S}(f)  \neq \emptyset $.

Baker\cite{bak9} and Iyer \cite{iye} investigated that if non-constant polynomial $ f $  permutes with transcendental entire function $ g $, then $ f(z) = e^{2m\pi i/k} z + b $ for some $ m, k \in \mathbb{N} $ and complex number b. That is, commuting polynomials  of any transcendental entire function $ f $ are from the group $ \Sigma (f) $ of symmetries of $ J(f) $.  On the other hand, some transcendental entire function can have a linear factor. For example, $ e^{z} + z, e^{e^{z}} + z, ze^{z}$ are transcendental entire functions have a linear factor. More generally, $ e^{az + b} + p(z)$, where $ p(z) $ is a non-constant polynomial, $ a(\neq 0) $   and $ b $ are two complex constants has a linear factor. Note that such type functions are known as \textit{prime functions in the entire sense}. More generally, an entire function $ f $ is prime (left prime) in the entire sense if $ f(z) = g(h(z)) $ for some entire functions $ g $ and $ h $, then either $ g $ or $ h $ is linear ($g$ is linear  whenever $ h $ is transcendental). Note that if $ q(z) $ is periodic entire function of finite lower order and $ p(z) $ is a non-constant polynomial, then $ q(z) + p(z) $ is prime in the entire sense. $ e^{z} + p(z) $ and $ \sin z + p(z) $ are examples of prime functions. These functions are nice examples of of transcendental entire functions belongs to category (6) stated above and so any entire function $ g $   that commutes with such functions can have same Julia sets. More detail study of above stated class of transcendental entire functions as well as other related results  can be found in {\cite[Theorems 1, 2 and 3]{ng}}, {\cite[Lemma 2.1, 2.2 and Theorems 2.1]{poon}} {\cite[Theorems 1, 2, 3 and 4]{sing}} and  {\cite[Theorems 1, 2 and 3]{wang}}. Further, more recent analysis have been made by Benini, Rippon and Stallard {\cite[Theorems 1.1, 1.2 and 1.3]{ben}}.

The chief concern of this paper has to consider the converse question: When do two transcendental entire functions have the same Julia set? That is, for two transcendental entire functions $ f $ and $ g $, if $ J(f) =J(g) $, then what will be the proper relation between $f$ and $ g $? To prove theorem \ref{nb1}, we need the following lemma  which is analogous to {\cite [Theorem 4.2]{hin}} of rational semigroup. It proves there is a relation of virtually abelian between the transcendental entire functions that have the same Escaping set. Note that Julia set is the boundary of escaping set,  so whatever result hold for escaping set, the same type of result hold for Julia set. 

\begin{lem}\label{nb2}
let $ f $ and $ g $ be transcendental entire functions. If $ I(f) = I(g)  $, then there is a map $ \phi(z) =az + b $ with $ |a| =1 $ such that $ f \circ g = \phi \circ g \circ f $. 
\end{lem}
\begin{proof}
We prove this lemma on the basis of the sequence of the following facts.
\begin{enumerate}
\item Functions $ f $ and $ g $ as stated in the lemma, the following statements hold:
\begin{enumerate}
\item $ I(f) $ is completely invariant under $ g $.
\item $ J(f) $ is completely invariant under $ g $.
\item $ J(f) =J(g) $.
\item $ \mathcal{S}(f) $ is a semigroup. \\
From $ I(f) = I(g) $, (a) follows easily. (b) follows as $ \partial I(f) = J(f) $ (boundary of the completely invariant set under the same function is completely invariant (see for instance {\cite[Theorem 3.2.3]{bea}})). (c) follows from the given $ I(f) =I(g) $. For (d), let us suppose $ g_{1}, g_{2} \in \mathcal{S}(f)$. By (b) $ J(f) $ is completely invariant under both $ g_{1} $ and $ g_{2} $ and hence it is completely invariant under $ g_{1} \circ g_{2}$. By (c) $ J(f)  = J(g_{1} \circ g_{2})$. This proves $ g_{1}\circ g_{2} \in \mathcal{S} (f) $ and hence $ \mathcal{S} (f) $ is a semigroup. 
\end{enumerate}
\item For $ \phi \in \Phi $ and functions $ f $ and $ g $ as above, the following facts hold:
\begin{enumerate}
\item  $ \phi \circ g \in \mathcal{S}(f) $ and $ g\circ \phi \in \mathcal{S}(f) $.
\item  $ f\circ g\in  \mathcal{S}(f) $,  $ g\circ f\in  \mathcal{S}(f) $ and $ \phi \circ g \circ f \in  \mathcal{S}(f) $.\\
From 1(b), $ J(f) $ is completely invariant under $ g $ and $ \phi $ is a symmetry of $ J(f) $, so $ J(f) $ is completely invariant under both $\phi \circ g$ and $g\circ \phi  $. By 1(c), $ J(f)  = J(\phi \circ g)  $ and $  J(f)  = J(g \circ \phi) $. This follows (a). Since $ J(f) $ is completely invariant under $ f $ and by 1(b),  it is also completely invariant under $ g $, and so it is completely invariant under $ f \circ g $ and $ g \circ f $.  As $ \phi $ is a symmetry of $ J(f) $, $ J(f) $ is completely invariant under $ \phi \circ g \circ f$. By 1(c), $ J(f \circ g) = J( \phi \circ g \circ f) $.  This proves (b). 
\end{enumerate}
\item $ f $, $ g $ and $ \phi $ as above, then
\begin{enumerate}
\item $\mathcal{S} (f) =\{g:  f \circ g =\phi \circ g \circ f\; \text{for some}\; \phi\in \Sigma  (f) \}$. 
\item $\mathcal{S} (f) =\{g:  f \circ g = g \circ f\}$. \\
Since $ f $ and $ g $ are transcendental entire functions, so from 2(b), $ J(f \circ g) = J( \phi \circ g \circ f) $ must imply $ f \circ g = \phi \circ g \circ f $. This follows (a). If $ \phi $ is a trivial symmetry, (b) follows. 
\end{enumerate}
\end{enumerate}
\end{proof}
\begin{proof}[Proof of the Theorem \ref{nb1}]
Since $ \phi $ is a symmetry of Julia set $ J(f) $, so $ \phi(J(f)) = J(f) $. If we apply $ \phi $ on $ F(f) = \mathbb{C} - J(f) $, we get $ \phi(F(f)) = F(f) $. The second part of the nearly abelian semigroup follows from lemma \ref{nb2}. 
\end{proof}

\section{Proof of the Theorems \ref{1q} and \ref{1r}}

Hinkkanen and Martin {\cite [Theorem 4.1]{hin}} proved that the Julia set of the nearly abelian rational semigroup is same as Julia set of each of its function.  Indeed, this is a generalization of the result of abelian rational semigroup that we prove in {\cite[Theorem 3.1]{sub2}}. 
It will be difficult to say the same  in general if we take abelian transcendental semigroup. That is, if we have abelian transcendental semigroup $ S $, it would not always $ J(S) = J(f) $ for all $ f \in S $.  It would be sometime in certain case, and one of the case was proved by K.K. Poon {\cite[Theorem 5.1]{poo}}.
\begin{theorem}
Let $ S =\langle f_{1}, f_{2}, \ldots f_{n} \rangle  $ is an abelian finite type transcendental semigroup. Then $ F(S) =F(f) $ for all $ f \in S $.  
\end{theorem} 
Indeed this result looks like extension  work of the following results of A. P. Singh and Yuefei Wang {\cite[Theorems 2, 3]{sing}} of classical transcendental dynamics.
\begin{theorem}
Let $ f $ and $ g $ are two permutable transcendental entire maps. If both $ f $ and $ g $ have no wandering domains, then $ J(f) = J(f \circ g) = J(g)$. 
\end{theorem}
\begin{theorem}\label{et1}
Let $ f $ and $ g $ are two permutable transcendental entire maps. If both $ f $ and $ g $ are of bounded type, then $ J(f) = J(f \circ g) = J(g)$. 
\end{theorem}
  
Our particular interest is how far the result of K.K.  Poon {\cite[Theorem 5.1]{poo}} can be generalized to nearly abelian transcendental semigroup. 
We prove the following more general  result in the case of nearly abelian transcendental semigroup $ S $. It is indeed, the converse of the theorem \ref{nb1}. 
\begin{lem}\label{3}
Let $S$ be a transcendental  semigroup. Then
\begin{enumerate}
\item $int(I(S))\subset F(S)\;\ \text{and}\;\ ext(I(S))\subset F(S) $, where $int$ and $ext$ respectively denote the interior and exterior of $I(S)$.  
 \item $\partial I(S) = J(S)$, where $\partial I(S)$ denotes the boundary of $I(S)$. 
\end{enumerate}
\end{lem}
\begin{proof}
 We refer for instance lemma 4.2 
and  theorem 4.3  of \cite{kum2}.
\end{proof}
Note that this lemma \ref{3} is a extension of Eremenko's  result \cite{ere},  $\partial I(f) = J(f)$ of classical transcendental dynamics to more general  semigroup settings. 

\begin{proof}[Proof of the Theorem \ref{1q}]
We prove $ I(S) = I(f) $ for all $ f \in S $ and the reaming results follows from lemma \ref{3}. 

By  {\cite[Theorem 1.3]{sub3}}, $ I(S) \subseteq I(g) $  for any $ g \in S $.  For opposite inclusion, suppose  $ z \in I(g) $, then by the definition of escaping set, $ g^{n}(z) \to \infty $  as $ n \to \infty $. 
Since semigroup $ S $ is nearly abelian so for all $ f, g\in S$ there is  $ \phi \in \Phi $ such that $ f \circ g = \phi \circ g\circ f  $. By induction, we easily get $ f \circ g^{n} = (\phi \circ g)^{n}\circ f  $.
Thus for any $ z \in I(g) $, we have $ (f \circ g^{n})(z) = ((\phi \circ g)^{n} \circ f)(z) = (h^{n}\circ f)(z)  = (h^{n}(f(z))$, where $ \phi \circ g = h \in S$. As $ g^{n}(z) \to \infty $, then $ h^{n}(z) \to \infty $ for all $ h \in S $. So, $ h^{n}(f(z)) \to \infty $ as $ n \to \infty $. This implies that $ f (z) \in I(h) = I(\phi \circ g) \subset I(\phi) \cup I(g)$.  This gives either $ f(I(g)) \subset I( \phi) $  or $ f(I(g)) \subset I(g)) $ for all $f,  g \in S $.  Therefore, $ z \in I(S) $.
Thus $ I(g) \subseteq I(S) $. Hence, $I(S) = I(g)$ for all $ g \in S $. 
\end{proof}

It is known that wandering domains do not exist in the case of rational functions and polynomials but transcendental entire functions may have wandering domains. However, in generalized settings of semigroups, rational (or polynomial) semigroups may have wandering domains ({\cite[Theorem 5.2]{hin}}). In this section, we investigate that there are transcendental semigroups that may have wandering domains. This investigation is possible via nearly abelian transcendental semigroups. In particular, the theorem \ref{1q} is useful to prove no wandering domain theorem of transcendental semigroups. 

\begin{proof}[Proof of the Theorem \ref{1r}]
Since $ S $ is a nearly abelian semigroup of transcendental entire functions, so by the theorem \ref{1q}, we have  $ F(S) = F(f) $ for any  $ f \in S $. Since each $f\in  S $  is a transcendental entire functions of finite or bounded type, so by the theorems $4.32$  and $4.33$  of \cite{hou}, the Fatou set $ F(f) $  has no wandering domain. So, $ F(S) $  has no wandering domain.
\end{proof}

\section{Classification of Stable Basins of Fatou Components}
Recall that a subsemigroup $ T $ of a holomorphic semigroup $ S $ is said to be of \textit{cofinite index} if there exists collection of exactly n - elements $ \{f_{1}, f_{2}, \ldots, f_{n} \}$ of $S^{1} = S \cup \{ \text{Identity} \} $ such that for any $ f\in S $, there is $ i\in \{1, 2, \ldots, n \} $ such that
\begin{equation}\label{1.2}
f_{i}\circ f \in T 
\end{equation} 
The smallest $ n $ that satisfies \ref{1.2} is called \textit{cofinite index} of $ T $  in $ S $.

By the theorem \ref{1r}. no wandering domains in the statement of Theorem 2.3 of \cite{sub4} can be replaced by nearly abelian semigroup as in the following result. 

\begin{theorem}\label{ss2}
 Let $ S $ be a nearly abelian transcendental semigroup generated by finite or bounded type transcendental entire functions.  Let $ U $ be any component of Fatou set. Then the forward orbit $ \{U_{f}: f\in S\} $ of $ U $ under $ S $ contains a stabilizer of $ U $ of cofinite index.
\end{theorem}

The concept of nearly abelian semigroups is quite useful for the classification of periodic component (stable basin) of the Fatou set $ F(S) $.
Recall that a stable basin $ U $ of a transcendental semigroup $ S $ is
\begin{enumerate}
\item  \textit{attracting} if it is a subdomain of attracting basin of each $ f\in S_{U} $
\item  \textit{supper attracting} if it is a subdomain of supper attracting basin of each $ f\in S_{U} $
\item  \textit{parabolic} if it is a subdomain of parabolic basin of each $ f\in S_{U} $
\item  \textit{Siegel} if it is a subdomain of Siegel disk of each $ f\in S_{U} $
\item  \textit{Baker} if it is a subdomain of Baker domain  of each $ f\in S_{U} $
\end{enumerate}

In classical transcendental iteration theory, the stable basin is one of the above types but not Hermann because a transcendental entire function does not have Hermann ring {\cite[Proposition 4.2] {hou}}.

We have proved the following  analogous statement of {\cite[Theorem 6.2] {hin}} in the case of transcendental semigroup.
\begin{theorem}
Let $ U $ be a  component of the Fatou set $ F(S) $ of the nearly abelian  semigroup $ S $ of transcendental entire functions of finite or bounded type and $ V $ be a subset $ F(S) $ containing in the forward orbit of $ U $. Then $ V $ is attracting, super attracting, parabolic, Siegel  or Baker.                                                                                                                      
\end{theorem}
\begin{proof}
By the theorem \ref{ss2}, $ V \subset F(S)$ is a stable basin of cofinite index, so, $ V_{f} = V $ for some $ f \in S $. By the theorem \ref{1q}, $ F(S) = F(f) $ for any $ f \in S $. By the theorem 4.5 (1, 2, 3, 5) of \cite{hou}, $ V $ is a attracting (or supper attracting), parabolic, Siegel or Baker domain of $ F(f) = F(S) $.
\end{proof}
\begin{rem}
\begin{enumerate}
\item  In above theorem, Baker domain does not exist if $ S $ transcendental semigroup of bounded type. The fact is obvious. That is, for transcendental nearly abelian semigroup $ S $ of bounded type, We have $ F(S) = F(f) $ for any $ f \in S $. However, for bounded type transcendental entire function $ f $, Fatou set $F(f)$ does not contain Baker domain {\cite [Theorem 4.29] {hou}}.
\item Under the same condition as above, all components of $ F(S) $ are simply connected. The fact is obvious. As above,  we have $ F(S) = F(f) $ for any $ f \in S $. From {\cite[Proposition 3]{ere1}}, all components of $ F(f) $ are simply connected if $ f \in \mathscr{B} $.
 \end{enumerate}
\end{rem}

In the classical transcendental dynamics, if  stable domains of a transcendental entire function $ f $ are bounded, then the Fatou Set $ F(f) $ does not contain asymptotic values. This fact holds good in transcendental semigroup if it is nearly abelian. 

\begin{theorem}
If all stable domains of a nearly abelian transcendental semi-group $ S $ are bounded, then Fatou set $ F(S) $ does not contain any asymptotic values.
\end{theorem}
\begin{proof}
By the theorem \ref{1q}, $F(S) = F(f)$ for any $ f \in S $.  By the theorem 4.16 of \cite{hou}, $F(f)$  does not contain any asymptotic value of $ f $. Hence the result follows.
\end{proof}

\end{document}